\documentclass[11pt,oneside,english,reqno]{amsart}
\usepackage[T1]{fontenc}
\usepackage[latin9]{inputenc}
\usepackage{color}
\usepackage{babel}
\usepackage{units}
\usepackage{amsthm}
\usepackage{amstext}
\usepackage{amssymb}
\usepackage{setspace}
\usepackage[unicode=true,pdfusetitle,
 bookmarks=true,bookmarksnumbered=false,bookmarksopen=false,
 breaklinks=false,pdfborder={0 0 0},backref=false,colorlinks=true]
 {hyperref}
\usepackage{breakurl}

\makeatletter

\newcommand{\noun}[1]{\textsc{#1}}

  \theoremstyle{plain}
  \newtheorem*{thm*}{\protect\theoremname}
\theoremstyle{plain}
\newtheorem{thm}{\protect\theoremname}
  \theoremstyle{plain}
  \newtheorem{lem}[thm]{\protect\lemmaname}
  \theoremstyle{definition}
  \newtheorem{defn}[thm]{\protect\definitionname}
 \theoremstyle{definition}
 \newtheorem*{defn*}{\protect\definitionname}

\usepackage{calrsfs}
\usepackage{graphicx}
\usepackage{color}
\usepackage{perpage}
\usepackage{upquote}
\usepackage{bbm}
\usepackage[shortlabels]{enumitem}

\MakePerPage{footnote}
\gdef\SetFigFontNFSS#1#2#3#4#5{} 
\gdef\SetFigFont#1#2#3#4#5{} 
\def\clap#1{\hbox to 0pt{\hss#1\hss}}

\DeclareMathOperator{\Sch}{Sch}

\DeclareMathOperator{\dist}{dist}
\DeclareMathOperator{\supp}{supp}

\DeclareMathOperator{\Res}{Res}
\DeclareMathOperator{\Cay}{Cay}

\definecolor{myblue}{rgb}{0.09,0.32,0.44} 
\hypersetup{pdfborder={0 0 0},pdfborderstyle={},colorlinks=true,linkcolor=myblue,citecolor=myblue,urlcolor=blue,breaklinks=true}

\theoremstyle{remark}
\newtheorem*{qst*}{Question}
\newtheorem*{rmrks*}{Remarks}

\newlength{\tempindent}
\newcommand{\lazyenum}{
\setlength{\tempindent}{\parindent}
\begin{enumerate}[leftmargin=0cm,itemindent=0.7cm,labelwidth=\itemindent,labelsep=0cm,align=left,label=\arabic*)]
\setlength{\parskip}{\smallskipamount}
\setlength{\parindent}{\tempindent}
}
\allowdisplaybreaks

\makeatother

  \providecommand{\definitionname}{Definition}
  \providecommand{\lemmaname}{Lemma}
  \providecommand{\theoremname}{Theorem}
\providecommand{\theoremname}{Theorem}

\makeatletter
\let\@wraptoccontribs\wraptoccontribs
\makeatother

\begin{document}

\title{Groups with minimal harmonic functions as small as you like}

\author{Gideon Amir}
\address{
GA: Department of Mathematics, Bar-Ilan University
Ramat Gan 52900, Israel.
e-mail: \tt{gidi.amir@gmail.com}
}
\author{Gady Kozma}
\address{
GK: Department of Mathematics, Weizmann Institute, Rehovot, 76100, Israel. e-mail: \tt{gady.kozma@weizmann.ac.il}
}
\address{
NMB: Department of Mathematics, ETH Zurich, Ramistrasse 101, 8092 Zurich, Switzerland. e-mail: \tt{nicolas.matte@math.ethz.ch}
}
\contrib[with an appendix by]{Nicol\'as Matte Bon}

\begin{abstract}
For any order of growth $f(n)=o(\log n)$ we construct a finitely-generated group $G$ and a set of generators $S$ such that the Cayley graph of $G$ with respect to $S$ supports a harmonic function with growth $f$ but does not support any harmonic function with slower growth. The construction uses permutational wreath products $\nicefrac{\mathbb{Z}}{2}\wr_{X}\Gamma$, in which the base group $\Gamma$ is defined via its properly chosen action on $X$.
\end{abstract}

\maketitle

\section{Introduction}

A harmonic function on a graph is a function $f$ from its vertices
into $\mathbb{R}$ such that for every vertex $v$, $f(v)$ is equal
to the average of $f$ over all the neighbours of $v$. Finite connected
graphs have no nonconstant harmonic functions, and infinite ones have,
in many interesting examples, surprisingly small families of harmonic
functions. For example, on the graph $\mathbb{Z}^{d}$, the only harmonic
functions with polynomial growth are polynomials \cite{H49}.
When the graph is the Cayley graph of some finitely generated group,
it is natural to try to relate properties of harmonic functions to
properties of the group. Graphs for which there are no nonconstant
bounded harmonic functions are called Liouville graphs, and they are
deeply connected with random walk entropy and amenability \cite{A76,D80,KV83,EK10,B+15}.
It is well-known (see e.g.\ \cite{T}) that any infinite Cayley graph supports harmonic functions of linear growth.
In a different regime, harmonic functions with linear growth were
used by Kleiner to give a new proof of Gromov's famous polynomial
growth theorem \cite{K10,T}.

There is an interesting quantitative version of the Liouville question
which goes as follows: for a given Cayley graph, what is the largest
$f:\mathbb{N}\to[0,\infty)$ such that any harmonic function $h$
with $h(x)=o(f(\dist(1,x)))$ is constant? (1 is of course the identity
element of the group, and $\dist$ is the graphical distance in the
Cayley graph). This question was addressed in \cite{B+} where a number
of examples where analysed. In particular, for the two dimensional
lamplighter group $\nicefrac{\mathbb{Z}}{2}\wr\mathbb{Z}^{2}$ it
was shown that it supports a harmonic function with logarithmic growth,
but that any $h$ with $h(x)=o(\log(\dist(1,x)))$ is constant. Surprisingly,
though, it turns out that this cannot be changed by using more complicated
lamps. Indeed, it was shown in \cite{B+} that for
\[
G=(\dotsb(\mathbb{Z}^{2}\wr\mathbb{Z}^{2})\wr\mathbb{Z}^{2})\dotsc\wr\mathbb{Z}^{2}
\]
the same behaviour holds, namely, any sublogarithmic harmonic function
must be constant (in sharp contrast to the behaviour of random walk
return probabilities and entropy on $G$).

In this paper we construct examples of groups with any $f$ between
$\log$ and constant. Here is the precise statement
\begin{thm*}
Let $f$ be a positive $C^{1}$ function on $[1,\infty)$ such that
$f(x)\to\infty$ and such that $xf'(x)$ is decreasing. Then there
exists a finitely generated group $G$ and a finite, symmetric set
of generators $S$ such that the Cayley graph $\Cay(G;S)$ has the
following properties:
\begin{enumerate}
\item There exists a nonconstant harmonic function $h$ on $\Cay(G;S)$
such that $|h(x)|\le Cf(\dist(1,x))$.
\item Any harmonic function $h$ on $\Cay(G;S)$ with $h(x)=o(f(\dist(1,x))$
is constant.
\end{enumerate}
\end{thm*}
It is a famous open problem whether the Liouville property is a group
property, i.e.\ whether it is possible for the same group to have
two sets of generators with respect to which one Cayley graph is Liouville
and the other is not. (It is known that the Liouville property is
not \emph{quasi-isometrically invariant} for general graphs, see \cite{L87,B91}).
In our case too, we do not know whether the minimal growth rate of
a harmonic function is a group property or it might depend on the
generators. For the \emph{specific} groups we are constructing, though,
it is possible to show that any set of generators has the same minimal
growth of harmonic functions. We will not prove it, though, as it
only adds technical complications to the proof.

Let us warn the reader against confusing a harmonic function of minimal
growth with a minimal harmonic function. A positive harmonic function $h$
is called minimal if any other positive harmonic function $g$ with $g(x)\le h(x)$
for all $x$ is a multiple of $h$ by a constant. Such functions play
a role in the construction of the Martin boundary of a group. Surprisingly,
perhaps, minimal harmonic functions in fact grow very fast. For example,
in $\nicefrac{\mathbb{Z}}{2}\wr\mathbb{Z}$ they actually grow exponentially
fast. A minimal harmonic function $h$ on $\nicefrac{\mathbb{Z}}{2}\wr\mathbb{Z}$
will have a very specific ``direction'' in which it decays exponentially
fast and this prevents any other harmonic function from being smaller
than $h$ \emph{everywhere}, but in a typical direction $h$ will
increase exponentially. We will not prove these claims as they are
somewhat off topic, but they follow in a more-or-less straightforward
manner from the description of minimal harmonic functions using Busemann
functions with respect to the Green metric, see \cite{BHM11}.

Our construction uses \emph{permutational wreath products} --- an
approach with a very successful track record in constructing groups
with interesting behaviour --- but with the following twist. A permutational
wreath product (exact definitions will be given below) starts from
a group acting on a set $X$. In most constructions so far the groups
were automaton groups, and these act naturally on certain sets such
that the result is a ``graphical fractal''. These groups and their
actions have been studied extensively, and the construction requires
deep knowledge of this theory. Here, instead of starting with a well-studied
group, we start with a graph (which we also denote by $X$, the group
will act on its vertices). We colour the edges of the graph such that
any vertex is incident to all colours, and then use the colouring
to construct a group acting on $X$. Each colour (say azure) will
correspond to the permutation of $X$ given by taking every vertex
$x$ to the vertex on the other side of the edge $e$ incident to
$x$ and coloured azure. The group will simply be the group of permutations
generated by all the colours. We shall see that if the colouring is
chosen to have enough repetitions, then the group may be analysed
directly and quite simply, with no need for the heavy combinatorial
analysis typically associated with automaton groups. While not a truly
different technique, more of a different way to think about existing
techniques, we believe it is useful, and in fact it has already been
used in \cite{KV}.

\section{Preliminaries}

\subsection{Graph and random walk preliminaries}

For a graph $G$ and two vertices $x$ and $y$ we denote by $x\sim y$
the case where $(x,y)$ is an edge of $G$, and say that $x$ and
$y$ are neighbours. By $d(x,y)$ or $\dist(x,y)$ we denote the graphical distance
between them i.e.\ the length of the shortest path between them in
the graph (if one exists; $\infty$ otherwise). We denote by $B(x,r)$
the closed ball $B(x,r)=\{y:d(x,y)\le r\}$. If we need to stress that this is taken in some graph $G$, we shall denote the ball by $B_G(x,r)$. The sphere will be denoted by $\partial B$ i.e.\ $\partial B(x,r)=\{y:d(x,y)=r\}$. We shall also use $G$
to denote the set of vertices of $G$, so $x\in G$ means that $x$
is a vertex of $G$. If we need the set of edges of $G$, we shall
denote it by $E(G)$. For two graphs $G$ and $H$ we denote by $G\times H$
the standard graph product i.e. the graph with vertex set $\{(g,h):g\in G,h\in H\}$
and with $(g,h)\sim(g',h')$ if and only if $g=g'$ and $h\sim h'$
or $g\sim g'$ and $h=h'$.

The laplacian $\Delta$ of a graph $G$ is the operator on $\ell^{2}(G)$
defined by
\[
(\Delta f)(x)=\sum_{y\sim x}\left( f(x)-f(y)\right) .
\]
For a graph $G$ the simple random walk on $G$ is the stochastic
process on the vertices of $G$ which, whenever it is in some vertex
$x$, moves to any neighbour of $x$ with equal probability. For an
$x\in G$ and an $A\subset G$, the hitting time of $A$ (from $x)$
is the random time
\[
\min\{t\ge0:R(t)\in A\}
\]
where $R$ is simple random walk on $G$ with $R(0)=x$. If $R$ never
hits $A$, we consider the minimum to be $\infty$.

For a graph $G$ and two sets $A$, $B\subset G$, we define $\Res(A,B)$ to be the electrical resistance between them, i.e.\ construct an electrical network where each edge has resistance $1$, and where the sets $A$ and $B$ are fused to be one point each, and then measure the resulting effective resistance between these two points. For a formal definition see, say, the book \cite{DS84}. We shall consider the effective resistance between a point and the boundary of a ball around it with radius $R$, and the main property of effective resistance that we shall use is that this resistance is inversely proportional to the probability that a random walk starting from the point will hit the boundary of the ball before returning. See \cite[\S 1.3.4]{DS84} for more details. Other properties of electrical resistance will only be used to estimate the resistance in the Schreier graphs (Lemmas \ref{l:inf_seq} and \ref{lem:Res nvn}).

We denote by $C$ and $c$ constants, which might change from place
to place or even within the same formula. $C$ will denote constants
which are sufficiently large and $c$ constants which are sufficiently
small. The constants would depend only on the graph at hand (usually denoted by $T$) unless otherwise specified.
The notation $X\approx Y$ is short for $cX\le Y\le CX$.

\subsection{Group preliminaries: Cayley and Schreier graphs, permutational wreath
products}

Let $X$ be a set and let $G$ be a group acting on $X$ from the
right (denoted by $X\curvearrowleft G$). We denote the action of
a $g\in G$ on an $x\in X$ by $x.g$ (so, of course, $x.gh=(x.g).h$).
For a (finite) subset $S\subset G$, the right Schreier graph $\Sch(X;S)$
is the graph whose vertices are $X$ and whose edges are all $(x,x.s)$
for all $x\in X$ and $s\in S$. The (right) Cayley graph $\Cay(G;S)$
is the Schreier graph of the action of $G$ on itself by right multiplication.

The wreath product $\nicefrac{\mathbb{Z}}{2}\wr_{X}G$ is the group
\[
\left(\nicefrac{\mathbb{Z}}{2}\right)^{X}\rtimes G
\]
where the action of $G$ on $(\nicefrac{\mathbb{Z}}{2})^{X}$ implicit
in the notation $\rtimes$ is $g(\omega)(x)=\omega(x.g)$. In other
words, the product on $\nicefrac{\mathbb{Z}}{2}\wr_{X}G$ is
\[
(\omega,g)\cdot(\omega',g')=(\omega+\omega'(\cdot.g),gg').
\]
The action $X\curvearrowleft G$ induces an action $X\curvearrowleft\nicefrac{\mathbb{Z}}{2}\wr_{X}G$
by having the $G$ component act on $X$ and the $\omega$ component
doing nothing at all.
\begin{lem}
\label{lem:a to h}Let $X\curvearrowleft G$ and let $S\subset G$
be finite. Let $o\in X$ and let $a:X\to\mathbb{R}$ satisfy $\Delta a=\delta_{o}$,
where $\Delta$ is the laplacian of $\Sch(X;S)$ and $\delta_{o}$
is the Kronecker delta at $o$. Assume $a(o)=-\frac{1}{2}$. Then
the function
\[
f(\omega,g)=\omega(o)\cdot a(o.g)
\]
is (right) harmonic on $\nicefrac{\mathbb{Z}}{2}\wr_{X}G$ with respect
to the switch-or-move generators i.e.\ $\{(\delta_{o},1)\}\cup\{(0,s):s\in S\}$.
\end{lem}
Here and below, when we write $\omega(o)\cdot a(o.g)$ we consider
the group $\nicefrac{\mathbb{Z}}{2}$ to be embedded in $\mathbb{R}$
as $\pm1$ and $\omega(o)\cdot a(o.g)$ is then just multiplication
in $\mathbb{R}$. Some readers might benefit from thinking about $a$
as the \emph{harmonic potential }(see e.g.\ \cite{S76}) at $o$,
through we are not making any requirement that $a$ be minimal in
any sense.
\begin{proof}
This is a straightforward, if confusing, calculation. Denote by $\Delta_{\wr}$
the laplacian of $\nicefrac{\mathbb{Z}}{2}\wr_{X}G$ with respect
to the switch-or-move generators. We first examine $g$ for which
$o.g\ne o$. For such $g$ we have
\begin{align*}
(\Delta_{\wr}f)(\omega,g) & \stackrel{(*)}{=}f(\omega,g)-f(\omega+\delta_{o.g^{-1}},g)+\sum_{s\in S}f(\omega,g)-f(\omega,gs)\\
 & \stackrel{\textrm{\clap{\mbox{\ensuremath{{\scriptstyle (**)}}}}}}{=}\omega(o)\cdot a(o.g)-\omega(o)\cdot a(o.g)+\sum_{s\in S}\omega(o)\cdot a(o.g)-\omega(o)\cdot a(o.gs)\\
 & \stackrel{\textrm{\clap{\mbox{\ensuremath{{\scriptstyle (***)}}}}}}{=}\omega(o)\cdot(\Delta a)(o.g)=0.
\end{align*}
where $(*)$ follows from the definitions of $\nicefrac{\mathbb{Z}}{2}\wr_{X}G$
and $\Delta_{\wr}$; $(**)$ follows from the definition $f$ and
from the fact that if $o.g\ne o$ then $o.g^{-1}\ne o$ and then $(\omega+\delta_{o.g^{-1}})(o)$=$\omega(o)$;
and in $(***)$, $\Delta$ stands for the laplacian on $\Sch(X;S)$.

In the case that $o.g=o$ we get a similar formula, except that for
the generator $(\delta_{o},1)$ we have $(\omega,g)(\delta_{o},1)=(\omega+\delta_{o},g)$
which inverts the sign at $o$. Hence
\begin{align*}
(\Delta_{\wr}f)(\omega,g) & =\omega(o)\cdot a(o)+\omega(o)\cdot a(o)+\sum_{s\in S}\left(\omega(o)a(o)-\omega(o)a(o.s)\right)\\
 & =\omega(o)\cdot\big(2a(o)+(\Delta a)(o)\big)=0
\end{align*}
since we assumed $\Delta a(o)=1$ and $a(o)=-\frac{1}{2}$.
\end{proof}

\section{Spherically symmetric trees}

As explained in the introduction, we shall construct a group from a
colou\-red graph, and properties of the random walk on that graph will
allow us to infer properties of the group. We now reveal the nature
of the graph in question: it will be a product of a spherically symmetric
tree with $\mathbb{Z}$. We therefore start with some properties of
spherically symmetric trees and random walks on them.

A tree is a graph with no cycles. For a tree $T$ and a marked vertex
$o$ (also called the root of $T$), we say that $T$ is spherically
symmetric (with respect to $o$) if the degree of a vertex depends
only on its graph distance from $o$. For a spherically symmetric
tree $T$ with all degrees of all vertices $2$ or $3$ (except the
root which can have degree $1$ or $2$), we call the points of degree
$3$, branch points. The root of $T$ is considered a branch point
if its degree is $2$.
\begin{defn}
Let $\mathcal{T}$ be the family of spherically symmetric trees with
degrees as in the previous paragraph such that the distances
$b_{i}$ of the branch points from $o$ satisfy $\inf_{i}b_{i+1}/b_{i}>2$.\end{defn}
\begin{lem}
\label{lem:hitting r}Let $T\in\mathcal{T}$. Then for any $r\in\mathbb{N}$
and any $x\in B(o,r)$, the expected hitting time of $o$ from $x$ in the
graph $B(o,r)$ is $\le Cr^{2}$.\end{lem}
\begin{proof}
Let $H$ be the graph given by taking $B(o,r)$ and ``projecting
it on $\mathbb{Z}$'' i.e.\ for every $h\le r$, identifying all
the vertices with distance $h$ from $o$ (so $H$ would have multiple
edges). By spherical symmetry, the hitting time of $o$ in $H$ is
exactly as in $B(o,r)$. To estimate the hitting time in $H$ we use
the commute-time identity (see \cite{C+96}). Let
\[
\ell=\max\{2^{i}:b_{i}<r\}
\]
i.e.\ the number of branches of the tree at $r$ (in particular we
define $\ell=1$ in the case that $b_{1}\ge r$). Hence $|E(H)|\le\ell r$.
The resistance between $x$ and $o$ can be bounded above by the resistance
to the leaves, for which we have the formula
\[
\Res(o,\textrm{leaves})=b_{1}+\sum\frac{\min(b_{i+1},r)-b_{i}}{2^{i}}\le C\frac{r}{\ell}
\]
(where $C$ depends only on $\min b_{i+1}/b_{i}$). We get that the
commute-time is $\le Cr^{2}$ which of course bounds the hitting time
in $H$ and hence also in $B(o,r)$.\end{proof}
\begin{lem}
\label{lem:arnack}Let $T\in\mathcal{T}$. Then $T\times\mathbb{Z}$
satisfies a ``spherically symmetric Harnack inequality'' i.e.\ there
exists $C$ such that for any $x\in T\times\mathbb{Z}$, any $r>0$
and any $h$ which is positive harmonic in $B(x,2r)$ and spherically
symmetric, $\max_{B(x,r)}h\le C\min_{B(x,r)}h$.\end{lem}
\begin{proof}
For any $w\in T$ and any $r$ let $L(w,r)$ be the following subgraph
of $T$: we take all $v\in T$ such that $|d(o,v)-d(o,w)|\le r$,
examine the induced subgraph of $T$ and take the connected component
of $w$. See figure \ref{fig:Lxr}.
\begin{figure}
\input{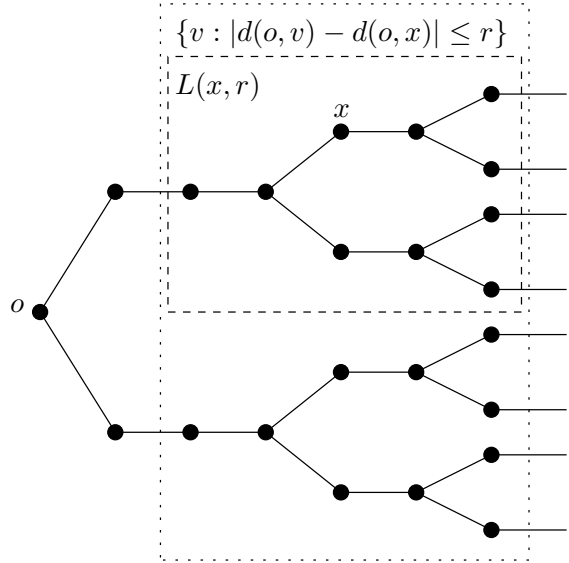}

\protect\caption{The subtree $L(w,r)$\label{fig:Lxr}}
\end{figure}

We first want to bound the second eigenvalue $\lambda_{L}$ of the
Laplacian on $L$. We note that $L$ is itself a tree and satisfies
the conditions of lemma \ref{lem:hitting r}. Hence the hitting time
of its root is bounded by $Cr^{2}$. This is well-known to bound the
mixing time of $L$, say by the equivalence of the mixing time and
the forget time (see \cite{LV97}) or by \cite[Corollary 1.2]{PSSS}.
Since the mixing time bounds the inverse of the second eigenvalue,
we get that $\lambda_{L}\ge cr^{-2}$.

Next we examine $Q=L\times\{-r,\dotsc,r\}$. Since the eigenvalues
of a graph product are simply sums of all couples of eigenvalues of
the two factors, we get a similar estimate for the second eigenvalue
of $Q$ i.e.\ $\lambda_{Q}\ge cr^{-2}$.

We now rewrite this inequality in a functional form that resembles
the Poincar\'e inequality. By this we mean
\begin{equation}
\sum_{x\in Q}m(x)|f-f_{Q}|^{2}\le Cr^{2}\sum_{(x,y)\in E(Q)}m(x,y)|f(x)-f(y)|^{2}\label{eq:PoincareQ}
\end{equation}
where $m(x,y)$ is 1 when $x\sim y$ and 0 otherwise, $m(x)=\sum_{y\in Q}m(x,y)$,
$E(Q)$ is the set of edges of $Q$ and $f_{Q}$ is the average of
$f$ over $Q$.

We now ``flatten'' $T\times\mathbb{Z}$. By this we mean that we
define a new weighted graph $F$. As a graph $F=\mathbb{Z}^{+}\times\mathbb{Z}$
and we define $\pi:T\times\mathbb{Z}\to F$ by $\pi(x,n)=(d(x,o),n)$.
We consider $\pi$ also as a map from $E(T\times\mathbb{Z})$ to $E(F)$.
We then define the weight of each edge $e$ in $F$ to be $|\pi^{-1}(e)|$
(we denote the weights in $F$ by $m(x,y)$ and $m(x)$ too). We remark
that $F$ is not a product graph itself. Inequality (\ref{eq:PoincareQ})
translates to a Poincar\'e inequality for squares in $F$. Indeed,
a square of $F$ is lifted to a disjoint collection of copies of $Q$.
Hence to show the Poincar\'e inequality we take a function $f$ on
the square, lift it (i.e.~consider $f\circ\pi$) to said disjoint
collection, use (\ref{eq:PoincareQ}) for each copy and sum.

Next we show that $F$ satisfies the volume doubling condition i.e.
\[
m(B(x,2r))\le Cm(B(x,r))\qquad\forall x,r
\]
where $B(x,r)$ is a ball in $F$ (in the usual graph metric which
ignores the weights) and where for any set of vertices $S$, $m(S)=\sum_{x\in S}m(x)$.
The easiest way to see this is to note that $m(L(x,2r))\le Cm(L(x,r))$
(this is a straightforward calculation, though we note that it uses
the condition that the branching points $b_{i}$ satisfy $b_{i+1}/b_{i}\ge2$),
conclude the same for $Q$ (this property is preserved by graph products),
and then by flattening for squares in $F$. Moving from squares to
balls follows by inscribing $B(x,2r)$ in a rectangle of side length
$2r+1$ and inscribing in $B(x,r)$ a rectangle of side length $r/2$.

Similarly we need to move the Poincar\'e inequality from squares
to balls, since the usual weak Poincar\'e inequality is stated for
balls, namely
\[
\sum_{x\in B(z,r)}m(x)|f-f_{B(z,r)}|^{2}\le Cr^{2}\sum_{(x,y)\in E(B(z,2r))}m(x,y)|f(x)-f(y)|^{2}\qquad\forall z,r
\]
(it is called ``weak'' because the ball on the left has radius $r$
and on the right has radius $2r$). Nevertheless it is well-know that
the squares and balls version are equivalent under volume doubling
(see e.g.\ \cite[\S 5]{J86}, the setting there is a little different
but the proof is the same).

The purpose of all these manoeuvres was to be able to apply Delmotte's
theorem \cite{D99} which states that the weak Poincar\'e inequality
and volume doubling imply Harnack's inequality. Lifting to $T\times\mathbb{Z}$
gives a Harnack inequality for spherically symmetric functions.\end{proof}
\begin{lem}
\label{cor:a exists}Let $T\in\mathcal{T}$ and let $T\times\mathbb{Z}\curvearrowleft G$.
Then there exists a harmonic function $m$ on $\nicefrac{\mathbb{Z}}{2}\wr_{T\times\mathbb{Z}}G$
such that
\[
m(x)\le C\Res_{T\times\mathbb{Z}}(d(o,x))
\]
where $\Res_{T\times\mathbb{Z}}(n)$ is the electrical resistance
in $T\times\mathbb{Z}$ from $o$ to $\partial B(o,n)$.\end{lem}
\begin{proof}
Let $r$ be arbitrary, and let $a=a_{r}:B_{T\times\mathbb{Z}}(o,r)\to\mathbb{R}$
satisfy that $\Delta a$ restricted to $B(o,r-1)$ is $\delta_{o}$,
$a(o)=0$ and is constant on $\partial B(o,r)$. Here and below $o=o_{T\times\mathbb{Z}}=(o_{T},0)$.
Since these conditions define $a$ uniquely, it is spherically symmetric.

Let $s<r/2$ and examine random walk $X_{t}$ on $B(o,s)$ starting
from $o$. Let $\tau_s$ be the hitting time of $\partial B(o,s)$, $\tau_o^+$ the first return time to $o$ and $\tau=min(\tau_o^+,\tau_s)$.
 Since $a(X_{t})$ is a (bounded) martingale for $1\le t\le \tau$, we have
\begin{equation*}
\mathbb{E}(a(X_1)) = \mathbb{E}(X_\tau) = {\mathbb{P}(\tau=\tau_s)}\mathbb{E}(a(X(\tau_{s}))\, | \, \tau=\tau_s) + {\mathbb{P}(\tau=\tau_o^+)a(o)}
\end{equation*}
Now $a(o)=0$, $\mathbb{E}(a(X_1))=1$ since $\Delta a(o)=1$ and by the Markov property $\mathbb{E}(a(X(\tau_{s})) \, | \, \tau=\tau_s) = \mathbb{E}(a(X(\tau_{s})))$
Thus
\begin{equation}
\Res(s)=\frac{1}{\mathbb{P}(X\mbox{ hits }\partial B(o,s)\mbox{ before returning to }o)}=\mathbb{E}(a(X(\tau_{s})))\label{eq:Res h}
\end{equation}
By Harnack's inequality (lemma \ref{lem:arnack}), all values of $a$
on $\partial B(o,s)$ are equal up to constants, and hence are also
equal to the expectation in (\ref{eq:Res h}) up to constants. We
get
\[
a_{r}(x)\le C\Res(d(o,x))\qquad\forall d(o,x)\le\frac{1}{2}r.
\]
This means that we may take a pointwise converging subsequence $a_{r_{k}}$.
The limit $a_{\infty}$ satisfies $\Delta a_{\infty}=\delta_{o}$
and $a_{\infty}(x)\le C\Res(d(o,x))$. An appeal to lemma \ref{lem:a to h}
finishes the proof.
\end{proof}

\begin{lem}\label{lem:Res nvn}For any $T\in\mathcal{T}$,
\[
\Res_{T\times\mathbb{Z}}(n)\approx\sum_{k=1}^{n}\frac{1}{n\ell(n)}
\]
where $\ell(n)$ is the number of branches of $T$ at level $n$.\end{lem}
\begin{proof}
For the lower bound we use the fact that identifying vertices only
decreases the resistance (See e.g. \cite{DS84} \S 2.2). Let $S_{n}$ be the square
\[
S_{n}=\{(t,k)\in T\times\mathbb{Z}:d(o,t)=n,|k|\le n\mbox{ or }k=\pm n,d(o,t)\le n\}
\]
and note that $S_{N/2}\subset B_{T\times \mathbb{Z}}(o,N)$. Let $E_{n}$ be the number of edges between $S_{n-1}$ and $S_{n}$.
We identify all vertices in $S_{n}$, for all $n$, and get
\[
\Res_{T\times\mathbb{Z}}(N)\ge\Res_{T\times\mathbb{Z}}(o,S_{N/2})\ge\sum_{n=1}^{N/2}\frac{1}{E_{n}}.
\]
To estimate $E_{n}$ examine edges between the two different pieces
of $S_{n}$. The number of edges between the part $\{d(o,t)=n\}$
in $S_{n-1}$ and $S_{n}$ is $(2n+1)\ell(n).$ The number of edges
between the other part of $S_{n-1}$ is (twice) the size of the tree
up to level $n$ i.e.\ $2\sum_{k=1}^{n}\ell(k)\le2n\ell(n)$. The
difference between summing $1/E_{n}$ up to $N$ and up to $N/2$
is no more than a constant. This proves the lower bound.

For the upper bound we construct a unit flow from $(o,0)$ to $S_{n}$
(in fact only to the ``vertical'' part of $S_{n}$, but this is
not important), and use the fact that the energy of any unit flow from $(o,0)$ to $S_{n}$ is an upper bound on $\Res_{T\times\mathbb{Z}}(o,S_{N})$ (See e.g. \cite{DS84} \S 1.3.5).
Let $t\in T$ satisfy $d(o,t)=n$ and let $|k|\le n$.
Examine the line in $\mathbb{R}^{2}$ from $(0,0)$ to $(n,k)$. Discretize
it to a path $\gamma_{k}$ in $\mathbb{Z}^{2}$ which does not go
left, and such that any edge of $\gamma_{k}$ is distance $\le1$
from the line (if there is more than one way to do it, choose one
arbitrarily). Translate $\gamma_{k}$ to a collection of $\ell(n)$
paths in $T\times\mathbb{Z}$ as follows: assume the vertex $(a,b)\in\mathbb{Z}^{2}$
translates to a $(t,b)\in T\times\mathbb{Z}$. If the next edge in
$\gamma_{k}$ is vertical, translate $(a,b\pm1)$ (as the case might
be) to $(t,b\pm1)$. If it is horizontal, translate $(a+1,b)$ to
$(t',b)$ where $t'$ is one of the children of $t$ in $T$. This
gives $\ell(n)$ different paths in $T\times\mathbb{Z}$. Send $\frac{1}{(2n+1)\ell(n)}$
mass through each such path. This is the desired flow.

Let us now calculate the energy of this flow. Since $T\times\mathbb{Z}$
has bounded degree, we may examine flows through vertices instead
of through edges, losing only a constant. Examine therefore a vertex
$(t,b)$, and assume $|b|\le d(o,t)$, since otherwise the flow is
zero. To check how many paths go through it, note that the line in
$\mathbb{R}^{2}$ must pass no further than distance $2$ from it.
This restricts to an angle of opening $\le C/d(o,t)$ and hence there
are no more than $C/d(o,t)$ of the flow in such paths. Further, the
flow divides evenly between the different $\ell(d(o,t))$ branches
at this level, so the flow through each particular branch is $1/\ell(d(o,t))$
of the total flow. We get that the flow is bounded by $C/d(o,t)\ell(d(o,t))$.
Summing the squares gives
\[
\Res_{T\times\mathbb{Z}}(o,S_{N})\le\!\!\!\sum_{\substack{(t,b):\\
|b|\le d(o,t)\le N
}
}\!\!\!\frac{C}{d(o,t)^{2}\ell(d(o,t))^{2}}\le\sum_{n=1}^{N}\frac{C}{n\ell(n)}.
\]
This finishes the lemma.\end{proof}

\begin{lem}
\label{lem:everybody knows}Let $T\in\mathcal{T}$, let $x,y\in T\times\mathbb{Z}$ and let $r>d(x,y)$. Let $\tau$
be the hitting time of $y$ by a random walk on $T\times\mathbb{Z}$
started from $x$. Let $\tau'$ be an independent exponential variable
with expectation $r^{2}$. Then
\[
\mathbb{P}(\tau>\tau')\approx\frac{1}{\Res_{T\times\mathbb{Z}}(r)}
\]
where the constants implicit in the $\approx$ might depend on $x$
and $y$.\end{lem}
\begin{proof}Since the claim is trivial in the case that $T\times\mathbb{Z}$ is transient, we shall assume it is recurrent i.e.\ $\Res(r)=\Res_{T\times\mathbb{Z}}(r)\to\infty$ as $r\to\infty$. We may assume that $r$ is sufficiently large. Recall from the proof of lemma \ref{lem:arnack} that the flattening $F$ of $T\times\mathbb{Z}$ satisfies volume doubling and Poincar\'e's inequality. Hence, by Delmotte's theorem \cite{D99} random walk on $F$ satisfies Gaussian bounds
\begin{equation}\label{eq:gauss}
\mathbb{P}^x(R_t=y)\le \frac{Cm(y)}{m(B(x,\sqrt{t}))}\exp(-cd(x,y)^2/t)
\end{equation}
where $m$ is as in the proof of lemma \ref{lem:arnack} as well. In particular we may conclude $\mathbb{P}^x(d(x,R_t)>\lambda \sqrt{t})\le C\exp(-c\lambda^2)$. From this we get (using a simple dyadic decomposition of $[0,t]$) the same bound for $\max_{s\le t}d(x,R_s)$. This last bound extends from $F$ back to $T\times\mathbb{Z}$ since if $R=R[0,t]$ is a random walk on $T\times\mathbb{Z}$ and if the image of $R$ in $F$ stayed in a ball of radius $\lambda\sqrt{t}$ throughout the time interval $[0,t]$, then $R$ must have stayed in a ball of radius $C\lambda\sqrt{t}$.

Another fact we ask the reader to recall, this time from lemma \ref{cor:a exists}, is that the probability that random walk starting from $x$ hits $y$ before hitting $\partial B(y,s)$ is $\approx 1/\Res(s)$ (recall that all $\approx$ signs and all $c$ and $C$ may depend on the points $x$ and $y$).

With these two facts we first conclude the lower bound on the probability. With probability $\approx 1/\Res(r)$ the random walk reaches $\partial B(y,r)$ before hitting $y$. It then has probability $>c$ to walk $cr^2$ time without getting to distance bigger than $\frac 12 r$, which of course prevents it from returning to $y$ during this time interval. During this period of length $cr^2$ there is a positive probability that $\tau'$ will occur. Hence
\[
\mathbb{P}(\tau>\tau')>\frac{c}{\Res(r)}.
\]

For the other direction, we first note that
\[
\mathbb{P}\Big(\tau'<\frac{r^2}{\Res(r)}\Big)\le\frac{1}{\Res(r)}.
\]
Denote this event by $\mathcal{B}_1$. Next we sum \eqref{eq:gauss} over $B(y,s)$ for some $s$ to claim that, for random walk on our flattened graph $F$,
\[
\mathbb{P}^x\big(d(R_{r^2/\Res(r)},y)<s\big)
\le 
\frac{Cm(B(y,s))}{m(B(x,r/\sqrt{\Res(r)}))}.
\]
Examining only the $\mathbb{Z}$ direction gives $m(B(y,s))\le m(B(y,s'))\cdot (s/s')$ so
\[
\mathbb{P}^x\Big(d(R_{r^2/\Res(r)},y)<\frac{r}{\Res(r)^{3/2}}\Big)\le\frac{C}{\Res(r)}.
\]
This estimate extends from $F$ to $T\times\mathbb{Z}$ as distances in $T\times\mathbb{Z}$ are bigger than the distances in the projection to $F$. Denote this event by $\mathcal{B}_2$ and denote $s=r/\Res(r)^{3/2}$.

Thus if neither of $\mathcal{B}_1$, $\mathcal{B}_2$ occurred, then the random walk exited a ball of radius $s$ before time $\tau'$. Adding the estimate for the probability for exiting a ball of radius $s$ before $\tau$ mentioned above gives
\begin{equation}\label{eq:Res_r_Res}
\mathbb{P}(\tau>\tau')\le\frac{C}{\Res(r)}+\frac{C}{\Res(s)}.
\end{equation}
Applying lemma \ref{lem:Res nvn} twice gives
\[
\Res(s)\ge c\sum_{n=1}^s \frac{1}{n\ell(n)}\ge c\sum_{n=1}^r\frac{1}{n\ell(n)}-C\log(r/s)\ge c\Res(r)-C\log(r/s).
\]
We insert this into \eqref{eq:Res_r_Res} to get
\[
\mathbb{P}(\tau>\tau')\le \frac{C}{\Res(r)}+\frac{C}{\Res(r)-C\log\Res(r)}.
\]
The lemma is proved.
\end{proof}
\section{Proof of the theorem}

At this point we set out to construct an action of some group $G$
on a tree $T$ and we need $G$ to be ``small'' in a sense to be
prescribed later.
\begin{defn*}
Let $T\in\mathcal{T}$. Colour the edges of $T$ with three colours,
azure, bordeaux and chartreuse such that no two neighbouring edges
have the same colour. Edges between $b_{i}+1$ and $b_{i+1}$ will
be coloured azure and bordeaux alternatingly, with all edges of a
fixed distance from $o$ having the same colour, and chartreuse will
be used in the branching points for one of the children. See figure
\ref{fig:abc}.
\begin{figure}
\input{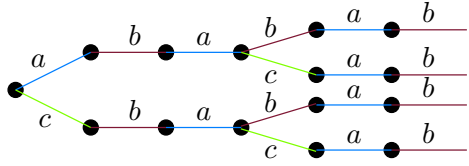}

\protect\caption{Colouring by azure, bordeaux and chartreuse, marked by $a$, $b$
and $c$ respectively for readers with a monochrome copy.\label{fig:abc}}
\end{figure}

Let $G$ be the subgroup of the group of permutations of $T$ (acting
from the right) generated by the following three involutions termed
$a$, $b$ and $c$. The generator $a$, for example, maps every vertex
$x$ to the vertex on the other end of the edge coloured azure containing
$x$, or, if no such edge exists, maps $x$ to itself (if you prefer,
add loops to all vertices of $T$ of degree $2$ and colour them in
the missing colour). We call $G$ the coloured involution group of
$T$.
\end{defn*}

\begin{defn*}
We define a subpolynomially growing tree is a $T\in\mathcal{T}$ such that $|B(o,r)|=r^{1+o(1)}$.\end{defn*}
\begin{lem}
\label{lem:counting balls}Let $T$ be a subpolynomially growing tree.
Then for every $r$ the number of different coloured graphs that can
arise as balls of radius $r$ in $T$ is $r^{1+o(1)}$.\end{lem}
\begin{proof}
Let $x\in T$ and examine the number of branch points in $B(x,r)$.
If there are none, then the ball is a line and there are exactly 2
possibilities for the colouring. If there is one, then there are $r+1$
possibilities for the distance of $x$ from the branch point, $3$
possibilities for the ``side'' of $x$ (father's side, chartreuse
child's side or non-chartreuse child's side), and two more possibilities
which depend on whether the non-chartreuse child of the branch point
is azure or bordeaux. All in all we get no more than $12r+12$ possibilities.

Finally, if there is more than one branch point, then we have two
branch points with distance less than $2r$ between them. Since $b_{i+1}/b_{i}>2$,
this means that both branch points are no further than $4r$ from
the root, and $x$ is no further than $5r$ from the root. Since $T$
is subpolynomially growing, the number of possibilities for this is
no more than $r^{1+o(1)}$.\end{proof}
\begin{lem}
\label{lem:entorpy}Let $T$ be a subpolynomially growing tree, and
let $G$ be its coloured involution group. Then $G$ satisfies that
$H_{n}\le n^{1/2+o(1)}$ where $H_{n}$ is the entropy of simple random
walk on the Cayley graph of $G$ with respect to the generators $a$,
$b$ and $c$. \end{lem}
\begin{proof}
Examine random walk $X_{n}$ on $G$. For every $x\in T$, $x.X_{n}$
is a simple random walk on $T$. By the Carne-Varopoulos theorem (see
\cite{C85,V85,P08} or \cite[\S 13.2]{LP}),
\[
\mathbb{P}(d(x.X_{n},x)>C\sqrt{n\log n})\le n^{-3}
\]
for some $C$ sufficiently large. By lemma \ref{lem:counting balls},
there are only $n^{1+o(1)}$ different balls we need to consider,
hence
\[
\mathbb{P}(\exists x\in T\mbox{ such that }d(x.X_{n},x)>C\sqrt{n\log n})\le n^{-2+o(1)}.
\]
Denote the event above by $\mathcal{B}$. We now divide $H_{n}$ as
follows
\[
H_{n}\leq \mathbb{P}(\mathcal{B})H(X_{n}\,|\,\mathcal{B})+\mathbb{P}(\mathcal{B}^{c})H(X_{n}\,|\,\mathcal{B}^{c}) + 1.
\]
For both terms, we now bound the entropy by the logarithm of the state
space. For $H(X_{n}\,|\,\mathcal{B})$ we bound it simply by $Cn$
while for $H(X_{n}\,|\,\mathcal{B}^{c})$ we have a smaller state
space: since there are only $(C\sqrt{n\log n})^{1+o(1)}=n^{1/2+o(1)}$
different relevant balls, and since each such ball has volume less
than $n^{1/2+o(1)}$, the state space is smaller than
\[
\big(n^{1/2+o(1)}\big)^{n^{1/2+o(1)}}
\]
and its logarithm is less than $n^{1/2+o(1)}$. Combining this with
$\mathbb{P}(\mathcal{B})\le n^{-2+o(1)}$, the lemma is proved. \end{proof}
\begin{thm}
\label{thm:almost}Let $T$ be a subpolynomially growing tree, and
let $G$ be its coloured involution group. Let $H=\nicefrac{\mathbb{Z}}{2}\wr_{(T\times\mathbb{Z})}(G\times\mathbb{Z})$
where $G\times\mathbb{Z}$ acts on $T\times\mathbb{Z}$ by
\[
(t,n).(g,m)=(t.g,n+m)\qquad\forall t\in T,g\in G,n,m\in\mathbb{Z}.
\]
Let $S=\{(a,0),(b,0),(c,0),(1,1),(1,-1)\}$ so that $S$ generates $G\times\mathbb{Z}$. Consider
$H$ as a Cayley graph with respect to the switch-or-move generators
$\{(\delta_{o},1)\}\cup\{(0,s):s\in S\}$ where $o$ is the root of
$T$. Then $H$ supports a harmonic function $m$ with
\[
m(x)\le\Res_{T\times\mathbb{Z}}(d(o,x))
\]
and any harmonic function $h$ with $h(x)=o(\Res_{T\times\mathbb{Z}}(d(o,x)))$
is constant.\end{thm}
\begin{proof}
The claim that $m$ exists follows from lemma \ref{cor:a exists}.
Let therefore $h$ be a harmonic function with $h(x)=o(\Res_{T\times\mathbb{Z}}(d(o,x)))$.
Recall that $x=(\omega,g,n)$ where $\omega\in\left(\nicefrac{\mathbb{Z}}{2}\right)^{T\times\mathbb{Z}}$
is the lamp configuration, $g\in G$ and $n\in\mathbb{Z}$. Our first
step is
\begin{lem}
\label{lem:copied shamelessly}$h$ does not depend on the lamp configuration
i.e.\ for every $\omega$, $\omega'$, $g$ and $n$, $h(\omega,g,n)=h(\omega',g,n)$. \end{lem}
\begin{proof}
We follow an argument from \cite{B+}. Assume first that $\omega$
differs from $\omega'$ in exactly one point of $T\times\mathbb{Z}$,
call it $v$. Examine lazy random walks $R$ and $R'$ on $H$ started
from $x=(\omega,g,n)$ and $x'=(\omega',g,n)$. Couple them as follows:
they walk exactly the same unless $v.R_{t}=o$ at some time $t$ (recall that $T\times\mathbb{Z}\curvearrowleft H$ via its $G\times\mathbb{Z}$ coordinate).
In this case if $R$ does a $(\delta_{0},1)$ let $R'$ do a lazy
step and vice versa (other kinds of steps are still done together).
In the former case we get $R_{t+1}=R'_{t+1}$ at this time, and we
then change the coupling rule so that they walk together for all time.

Fix some $r>0$ and examine our two coupled walks. Let $\tau'$ be
the stopping time when the two walks coupled as in the previous paragraph.
Let $\tau''$ be an exponential variable  with expectation $r^{2}$, independent of everything else.
Let $\tau=\min\{\tau',\tau''\}$. We now claim that
\begin{equation}
\mathbb{P}(\tau=\tau'')\le\frac{C}{\Res(r)}\label{eq:ooof this independent coin}
\end{equation}
where $C$ might depend on $x$ and $v$ but not on $r$. To see (\ref{eq:ooof this independent coin})
note that at every visit of $v.R_{t}=v.R_{t}'$ to $o$, there is
a positive probability that they will couple. Let $B_{k}$ be the
event that $v.R_{t}$ reached $o$ exactly $k$ times before time
$\tau''$, but no coupling occurred. On the one hand, there is probability
$e^{-ck}$ that in $k$ attempts, coupling always failed. On the other
hand, after the $k^{\textrm{th}}$ visit, $\tau''$ is independent
of the past, so $B_{k}$ implies that $\tau''$ happened before $v.R_{t}$
returned to $o$. By lemma \ref{lem:everybody knows}, because $v.R_{t}$
is a simple random walk on $T\times\mathbb{Z}$, this probability
is $\le C/\Res(r)$. We get
\[
\mathbb{P}(\tau=\tau'')\le\sum_{k=0}^{\infty}e^{-ck}\frac{C}{\Res(r)}
\]
the constant $C$ here comes from lemma \ref{lem:everybody knows},
so it depends on the starting point of the walk, but there are only
two starting points to consider, $v.(g,n)$ (for $k=0$) and $o$
(for $k\ge1$). This shows (\ref{eq:ooof this independent coin}).

Returning to our two coupled walks, since $h(R_{t})$ is a martingale,
\[
h(x)=\mathbb{E}(h(R_{\tau}))=\mathbb{E}(h(R_{\tau})\cdot\mathbf{1}\{\tau=\tau''\})+\mathbb{E}(h(R_{\tau})\cdot\mathbf{1}\{\tau\ne\tau''\})
\]
(justifying integrability is easy because $h(x)\le Cd(o,x)$ and $\tau$
is bounded by $\tau''$ which is an exponential variable). The same
holds for $h(x')$ with $R_{t}'$ instead of $R_{t}$. However, if
$\tau\ne\tau''$ then $R_{\tau}=R_{\tau}'$ so these terms give equal
contribution to $h(x)$ and $h(x')$. We get
\[
|h(x)-h(x')|\le|\mathbb{E}(h(R_{\tau})\cdot\mathbf{1}\{\tau=\tau''\})|+|\mathbb{E}(h(R_{\tau}')\cdot\mathbf{1}\{\tau=\tau''\})|.
\]
A crude bound over $\tau''$ will give
\[
|h(x)-h(x')|\le2\mathbb{P}(\tau=\tau'')\cdot\max\{|h(x)|:d(o,x)\le r^{3}\}+O(e^{-cr}).
\]
By (\ref{eq:ooof this independent coin}) and our assumption on $h$,
we get
\[
|h(x)-h(x')|\le\frac{o(\Res(r^{3}))}{\Res(r)}+O(e^{-cr}).
\]
Lemma \ref{lem:copied shamelessly} now follows from the next simple
claim.\end{proof}
\begin{lem}\label{l:inf_seq}
There exists an infinite sequence $r_{k}\to\infty$ such that
\[
\frac{\Res(r_{k}^{3})}{\Res(r_{k})}\le C.
\]
\end{lem}
\begin{proof}
Examine $\Res(2^{3^{k}})$. Because our graph contains $\mathbb{Z}^{2}$,
$\Res(2^{3^{k}})\le C3^{k}$, (See e.g. \cite{DS84} \S 2.2). Further, this
is an increasing sequence. Hence for infinitely many choices of $k$,
$\Res(2^{3^{k}})\le4\Res(2^{3^{k-1}})$.
\end{proof}
We return to the proof of theorem \ref{thm:almost}. We have just
established lemma \ref{lem:copied shamelessly}, i.e.\ that $h$
does not depend on the lamps. We get that $h(\omega,g,n)=h'(g,n)$,
and $h'$ is harmonic on $G\times\mathbb{Z}$. To show that it is
constant, we use our entropy estimate. Indeed, by \cite[theorem B]{EK10},
if for some $f$ and a sequence $n_{k}$ we have
\[
\lim_{k\to\infty}f(n_{k})\sqrt{H_{n_{k}+1}-H_{n_{k}}}=0
\]
then a harmonic function with growth at most $f$ is constant. By
lemma \ref{lem:entorpy}, the entropy of simple random walk on $G$
(with respect to the generators $a$, $b$ and $c$) is at most $n^{1/2+o(1)}$.
Thus the same holds for $G\times\mathbb{Z}$ (with the added generators
of $\mathbb{Z}$). Denote the entropy of $G\times\mathbb{Z}$ by $H_{n}$.
Thus one may find a sequence $n_{k}$ such that $H_{n_{k}+1}-H_{n_{k}}\le n_{k}^{-1/2+o(1)}$.
We take the function $f$ to be, say, $f(n)=n^{1/8}$ and get that
any harmonic function $h'$ on $G\times\mathbb{Z}$ with $h'(x)\le Cd(o,x)^{1/8}$
is a constant. Since our function $h'$ satisfies
\[
h'(x)=o(\Res_{T\times\mathbb{Z}}(d(o,x)))=o(\Res_{\mathbb{Z}\times\mathbb{Z}}(d(o,x))=o(\log d(o,x)),
\]
it must be constant, and theorem \ref{thm:almost} is proved.
\end{proof}

\subsection{Eligible growth rates}

Theorem \ref{thm:almost} is almost our stated result. We merely need
to demonstrate that for any function $f$ with $f(x)\to\infty$ and
$xf'(x)$ decreasing, one may construct a subpolynomially growing
tree $T$ such that $\Res_{T\times\mathbb{Z}}(n)\approx f(n)$. This
is no more than a calculus exercise, but let us do it in details nonetheless.
\begin{lem}\label{lem:f nvn}Let $f$ be a positive $C^{1}$ function on $[1,\infty)$
such that $f(x)\to\infty$ and that $xf'(x)$ is decreasing. Then
there exist $\ell(n)$ such that
\[
\sum_{n=1}^{N}\frac{1}{n\ell(n)}\approx f(N).
\]
and such that $\ell(n)$ can be taken to be the branching values of
a subpolynomially growing tree. The constant implicit in the $\approx$
may depend on $f$, but not on $N$. \end{lem}
\begin{proof}
We would have liked to define $\ell(n)=1/nf'(n)$ but this might not
satisfy the requirement that $\ell(2n)\le 2\ell(n)$, needed for a tree in
$\mathcal{T}$. Define therefore
\[
w(n)=\min_{k=0}^{\infty}\frac{4^{k}}{nf'(n2^{-k})}\qquad w_{2}(n)=\min\{w(n),\log^{2}8n\}
\]
where in the definition of $w$ we extend $f$ below $1$ to be $f(x)=f(1)+f'(1)\log x$,
an extension which preserves the condition that $xf'(x)$ be decreasing, and
ensures that the minimum is achieved. We note that $w(2n)\le2w(n)$.
Indeed, $w(n)=4^{k}/nf'(n2^{-k})$ for some $k$, and using $k+1$
in the definition of $w(2n)$ gives $w(2n)\le4^{k+1}/2nf'(2n\cdot2^{-k-1})=2w(n)$.
The same holds for $\log^{2}8n$ and hence also $w_{2}(2n)\le2w_{2}(n)$.
Let us now estimate $\sum1/nw(n)$. On the one hand we have
\begin{align*}
\sum_{n=1}^{N}\frac{1}{nw_{2}(n)} & \ge\sum_{n=1}^{N}\frac{1}{n\cdot\frac{1}{nf'(n)}}\ge f'(N)+\sum_{n=1}^{N-1}f(n+1)-f(n)\\
 & =f(N)-f(1)+f'(N)\approx f(N)
\end{align*}
where in the first inequality we used the definition of $w$ with
$k=0$ and in the second inequality the fact that $f'$ is also decreasing,
so $f'(n)\ge f'(\xi)=f(n+1)-f(n)$. For the other direction, . we
write
\begin{align*}
\sum_{n=1}^{N}\frac{1}{nw_{2}(n)} & =\sum_{n=1}^{N}\max\frac{1}{n\log^{2}8n},\max_{k=0}^{\infty}\frac{f'(n2^{-k})}{4^{k}}\\
 & \stackrel{(*)}{\le}\sum_{n=1}^{N}\frac{1}{n\log^{2}8n}+\sum_{k=0}^{\infty}\frac{1}{4^{k}}\sum_{n=1}^{N}f'(n2^{-k})\\
 & \le C+\sum_{k=0}^{\infty}\frac{1}{4^{k}}\Big(f'(2^{-k})+2^{k}\sum_{n=2}^{N}f(n2^{-k})-f((n-1)2^{-k})\Big)\\
 & \stackrel{\textrm{\clap{\ensuremath{{(**)}}}}}{\le}C+\sum_{k=0}^{\infty}\frac{1}{2^{k}}f(N2^{-k})\stackrel{(***)}{\le}C+2f(N)
\end{align*}
where in $(*)$ we estimated both maxima by a sum, and rearranged
the summands; $(**)$ follows from the fact that $f(x)=C+C\log x$
for $x<1$; and $(***)$ is due to the fact that $f$ is increasing
($f$ cannot decrease, since if at some $x$ we have $f'(x)<0$ then
$f'$ must be negative forever, contradicting the requirement $f(x)\to\infty$).
Defining
\[
\ell(n)=2^{\lfloor\log_{2}w_{2}(n)\rfloor}
\]
we see that $\sum\frac{1}{n\ell(n)}\approx f$ while at the same time,
$\ell(2n)\le 2\ell(n)$ which is the only condition necessary for $\ell(n)$
to be the branching numbers of a tree in $\mathcal{T}$. Since $\ell(n)\le\log^{2}8n$,
$T$ is also subpolynomially growing.
\end{proof}
Combining lemmas \ref{lem:Res nvn} and \ref{lem:f nvn} and theorem
\ref{thm:almost} proves our main result.

\subsection*{Acknowledgements}

We thank Michael Bj\"orklund for interesting discussions regarding
minimal harmonic functions.

While performing this research, G.A. was supported by the Israel Science Foundation grant \#1471/11. G.K. was supported by the Israel Science Foundation grant \#1369/15 and by the Jesselson Foundation.

\vfill

\newcommand{\Z}{\ensuremath{\mathbb{Z}}}%
\newcommand{\N}{\ensuremath{\mathbb{N}}}%

    \newcommand{\A}{\ensuremath{\mathcal{A}}}%
        \newcommand{\Rcal}{\ensuremath{\mathcal{R}}}%
                \renewcommand{\O}{\ensuremath{\mathcal{O}}}%

        \renewcommand{\G}{\ensuremath{\mathcal{G}}}%

        \newcommand{\M}{\ensuremath{\mathcal{M}}}%
                \renewcommand{\P}{\ensuremath{\mathcal{P}}}%
                \newcommand{\Pbb}{\ensuremath{\mathbb{P}}}%
                \newcommand{\Sbb}{\ensuremath{\mathbb{S}}}%
                \renewcommand{\S}{\ensuremath{\mathcal{S}}}%
                \newcommand{\Ccal}{\ensuremath{\mathcal{C}}}%

                \newcommand{\m}{\ensuremath{\overline{m}}}%

        \renewcommand{\H}{\ensuremath{\mathcal{H}}}%
                \newcommand{\K}{\ensuremath{\mathcal{K}}}%

  \newcommand{\Sym}{\ensuremath{\textrm{Sym}}}%
  \renewcommand{\supp}{\ensuremath{\textrm{Supp}}}%
	\newcommand{\fix}{\ensuremath{\textrm{Fix}}}%
  \newcommand{\gp}{\ensuremath{\textrm{germ}_+}}%
  \newcommand{\gm}{\ensuremath{\textrm{germ}_-}}%
    \newcommand{\acts}{\ensuremath{\curvearrowright}}%
  \newcommand{\sub}{\ensuremath{\operatorname{Sub}}}%
  \newcommand{\stab}{\ensuremath{\operatorname{Stab}}}%
  \newcommand{\homeo}{\ensuremath{\operatorname{Homeo}}}%
    \newcommand{\aut}{\ensuremath{\operatorname{Aut}}}%
    \newcommand{\st}{\ensuremath{\operatorname{St}}}%
    \newcommand{\rist}{\ensuremath{\operatorname{RiSt}}}%
    \newcommand{\cay}{\ensuremath{\operatorname{Cay}}}%

    \newcommand{\PSL}{\ensuremath{\operatorname{PSL}}}%

	\newcommand{\Pw}{\mathrm{P_W}}
	\newcommand{\PwG}{\mathrm{P_W}(G)}
	\newcommand{\autT}{\mathrm{Aut}(T)}
	
	\newcommand{\Cred}{C_{red}^\ast}
	\newcommand{\conj}{\mathcal{C}}
	\newcommand{\urs}{\ensuremath{\operatorname{URS}}}
		\newcommand{\sch}{\ensuremath{\operatorname{Sch}}}

\newcommand{\bigslant}[2]{{\raisebox{-.2em}{$#1$}\backslash\raisebox{.2em}{$#2$}}}

\newtheorem{prop}{Proposition}

\section*{\bf{Appendix: a remark on groups defined by slowly growing trees}}
\begin{center}
Nicol\'as Matte Bon
\end{center}

To prove their main result, the authors introduce a new idea to construct groups, defined by an explicit Schreier graph obtained by labelling a slowly growing tree, which is of independent interest.  We remark in this appendix that this construction yields elementary amenable groups (in fact, [locally finite]--by--dihedral, see Propositon \ref{F: first group} below). In particular the main result of this paper, which is pertinent to the realm of amenable groups, holds already in the more restricted realm of elementary amenable groups.  (Recall that the class of elementary amenable groups is the smallest class of groups that contain finite and abelian groups and is stable under  direct limits and extensions \cite{Chou}; certain types of asymptotic and algebraic behaviours are possible for amenable groups but not for elementary amenable groups).

A variant of this group construction, named \emph{bubble groups}, were introduced and used in \cite{KV} and further studied  in \cite{SCZ}, where their amenability was shown with an analytic method, and the isoperimetric profile of (permutational wreath products over) the bubble groups was studied and shown to realise a wide family of behaviors. Our remark also applies to the bubble groups (showing that they are locally finite--by--metabelian). It would be interesting to know if these group constructions can be used to study the behaviour of other asymptotic invariants among elementary amenable groups.

It is interesting to notice that these group constructions were somewhat inspired by previous ones based on automata groups, which are instead often non-elementary amenable.

\medskip

I thank the authors of this paper, Laurent Bartholdi, and Tianyi Zheng  for interesting conversations on the bubble groups, some related to the \emph{extensive amenability} of their natural action, which motivated in part this observation.
 \medskip

\subsection*{Definition of the groups} We work in parallel with two different groups belonging to two different families,   denoted $G$ and $\Gamma$, that we shall define precisely below. The group $G$ is essentially the same group denoted $G$ in Section 4, with the difference that we also allow the degree of branching points of the tree $T$ to vary (and to be possibly unbounded). The group $\Gamma$ is one of the \emph{bubble groups} from \cite{KV}, in a similar more general setting considered in \cite{SCZ}.

Choose and fix two sequences of positive integers: the \emph{scaling sequence} $(b_i)$, assumed to be strictly increasing, and  the \emph{degree sequence} $(d_i)$. We also assume that $d_i\geq 3$ for every $i\geq 1$. We set $d_\ast=\sup d_i$ (possibly $d_\ast=\infty$).

To define the group $G$, let $T$ be a spherically homogeneous rooted tree as in Figure 1, where the root has degree $d_0\geq 1$, and every other vertex has degree 2, except if its distance from the root is equal to $b_i$ for some $i>0$, in which case it has degree $d_i$. A vertex with degree greater than 2  will be called a \emph{branching point}.  We label the edges of $T$ using the letters $a,b$ and $c_n, n\in [1, d_\ast-2]$ as follows. Edges that are strictly between two branching points will be labelled by $a$ and $b$ alternatively, as in Figure 2. Edges adjacent to a branching point of degree $d_i$ will be labelled by $a, b, c_1, \ldots, c_{d_i-2}$.  Each letter $a, b$ or $c_n$ corresponds to a permutation of the vertex set of $T$,  denoted with the same letter, that exchanges the endpoints of every edge with the corresponding label. As in Section 4 we let $G$ be the group of permutations of $T$ generated by $a, b$ and $c_n, 1\geq n\geq d_\ast-2$. The group $G$ is finitely generated if and only if the degree sequence is bounded.
 Note that the groups appearing in the proof of the main theorem of the paper are of the form $\Z/2\wr_{T\times \Z}(G\times \Z)$ (see Theorem 10), for particular choices of the scaling sequence, and for a degree sequence constant and equal to 3. Therefore they are elementary amenable if and only if the group $G$ is so.

To define the bubble group $\Gamma$, we let $\Theta$ be the (oriented) graph obtained from the graph $T$ as follows.  Every branching point is replaced by an oriented cycle of length $d_i$, called a \emph{branching cycle}, and every path betwen a branching point at generation $i$ and a branching point at generation $i+1$ is replaced by a by an oriented cycle of length $2(b_{i+1}-b_i)$, called a \emph{bubble}; see Figure \ref{fig:bubble}. Edges on a bubble will be labelled by the letter $\alpha$ and edges on a branching cycle will be labelled by $\beta$. We still denote $\alpha, \beta$ the corresponding permutations, that permute cyclically every bubble or branching cycles respectively. The \emph{bubble group} is the group $\Gamma$  generated by $\alpha$ and $\beta$. Note that the group $\Gamma$, unlike the group $G$, is always 2-generated, even if the degree sequence $(d_i)$ is unbounded.
\begin{figure}
  \begin{center}
    \includegraphics[width=0.6\textwidth]{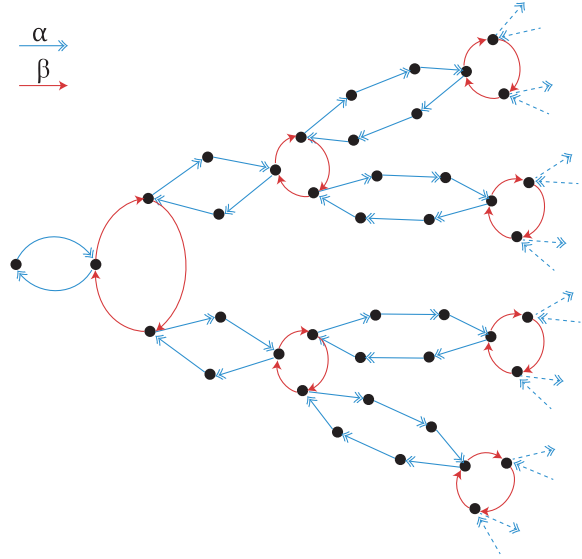}
  \end{center}
  \vspace{-10pt}
  \caption{The graph $\Theta$\label{fig:bubble}}
\end{figure}

\subsection*{Structure of \texorpdfstring{$\Gamma$}{Gamma} and \texorpdfstring{$G$}{G}.} In the statements below we make the assumption that  $b_{i+1}-b_i$ tends to infinity, while no assumption is made on $(d_i)$. Under the same assumptions, the fact the the group $\Gamma$ is amenable was first established in \cite[Proposition 5.13]{SCZ}.

\begin{prop} \label{F: first group} Let $(b_i), (d_i)$ be a scaling sequence and a degree sequence. Assume that $b_i-b_{i-1}$ tends to infinity, and let $(d_i)$ be arbitrary.  Then \begin{enumerate}
\item[(i)]
The group $G$ splits as a semi-direct product of the form
\[ G\simeq  N\rtimes D_\infty,\]
where $N$ is locally finite, and $D_\infty$ is the infinite dihedral group. The surjection $G\to D_\infty$ maps  $a, b$ to two generating involutions of $D_\infty$, and all the generators $c_j, j\geq 1$ to the identity.

\item[(ii)] The group $\Gamma$ is an extension of the form
\[\quad\quad 1\to K\to \Gamma\to C\wr \Z\to 1;\]
where $K$ is locally finite, and $C$ is a non-trivial cyclic group, which is finite if the degree sequence $(d_i)$ is bounded, and infinite cyclic if $(d_i)$ is unbounded.
The surjection $\Gamma \to C\wr \Z$ maps $\alpha$ to a standard generator of $\Z$, and $\beta$ to a standard generator of the lamp group over 0.
\end{enumerate}
In particular the groups $G$ and $\Gamma$ are elementary amenable.
\end{prop}

Before the proof, let us fix some terminology on spaces of Schreier graphs. If $H$ is a finitely generated group endowed with a finite generating set $S$, we denote $\sch(H, S)$ the space of   rooted, connected, oriented graphs with edges labelled by $S$, that arise as Schreier graphs of the pair  $(H, S)$, up to isomorphisms of rooted connected labelled graphs. Endowed with the topology induced by the space of marked graphs, $\sch(H,S)$ is a compact space, on which $H$ acts continuously by moving the root in the natural way. As it is well-known and easy to see, this action is conjugate to the conjugation action of $H$ on its space of subgroups $\sub(H)$ endowed with the Chabauty topology (however we will not need this point of view here).
  In our conventions, edges of a Schreier graphs are always oriented, with the exception that we represent edges corresponding to a generator of order two by a single unoriented one rather than two oriented ones (this is consistent with the fact that $T$ is not oriented, while $\Theta$ is). We systematically omit the loops when representing Schreier graphs. All Schreier graphs will be intended as \emph{rooted} and \emph{labelled} graph, and we will specify \emph{unrooted} when we forget the root.

 Proposition \ref{F: first group} follows by analysing the closure of the orbits of $T$ and $\Theta$ in the  spaces of Schreier graphs of $G$ and $\Gamma$. In one sentence, each graph in the closure provides a quotient of the groups, and these can be readily described.
\begin{proof}

Throughout the proof, we say that an integer $n\geq 3$ is \emph{admissible} if $d_i=n$ for infinitely many $n$s. For simplicity, let us first consider the case where the degree sequence $d$ is bounded.
We begin by proving (i) and then explain the modifications needed for (ii). Let $S=\{a, b, c_1, \ldots \}$ be the standard generating set of $G$. Let $o\in T$ be the root, and view $(T, o)$ as an element of $\sch(G, S)$. We look at the closure of the $G$-orbit of $(T, o)$ in $\sch(G, S)$. Each Schreier graph in the closure has underlying unrooted graph  of one of the following three types: the graph $T$, the graph $\overline{T}_n$  as in Figure \ref{fig:not-bubble} A, where the degree $n$ of its unique branching point runs over admissible integers, and the graph $\hat{T}$, shown in Figure \ref{fig:not-bubble} B.  Each of these labelled graphs naturally defines a group of permutations of its vertices, generated by the permutations of its vertices that correspond to each letter (as in Section 2). Since all these graphs are Schreier graphs of $(G, S)$, the group defined by each of them is a quotient of $G$. The group defined by $\hat{T}$ is simply the dihedral group $D_\infty$. Therefore we have a surjection $p\colon G\to D_\infty$ mapping $a, b$ to two generating involutions of $D_\infty$. Since $a,b$ already generate a subgroup isomorphic to $D_\infty$ in $G$, the group $G$ splits as a semi-direct product of the form $G\simeq\ker(p)\rtimes D_\infty$. We shall now check that $\ker(p)$  is locally finite. For each admissible $n$ we denote $\overline{G}_n$ the group of permutations of vertices of $\overline{T}_n$ defined by the labelling of edges of $\overline{T}_n$, and $\pi_n\colon G \to \overline{G}_n$ the associated natural surjection. Observe that $\hat{T}$ is also in the closure of the orbit of the rooted graphs whose underlying graph is $\overline{T}_n$, therefore it is also a Schreier graph of $\overline{G}_n$ for every $n$. Thereby we also have a surjection $p_n\colon \overline{G}_n\to D_\infty$. Clearly $\ker(p_n)$ is a locally finite subgroup of $\overline{G}_n$: in fact every element of $\ker(p_n)$ acts trivially sufficiently far from the branching point of  $\overline{T}_n$, hence $\ker(p_n)$ consists of permutations of $\overline{T}_n$ with finite support.  Moreover it is clear that $p=p_n\circ \pi_n$, as it is seen by looking at the images of generators. It follows that $\pi_n$ maps $\ker(p)$ inside $\ker (p_n)$ for every admissible $n$.  Consider the diagonal map $\pi: \prod_n \pi_n\colon G\to \prod_n\overline{G}_n$, where the product is taken over all admissible $n$s. It follows from the discussion above that $\pi$ maps $\ker(p)$ to the locally finite group $\prod_n \ker p_n$. Hence to check that $\ker(p)$ is locally finite, it is enough to check that $\ker \pi=\cap_n \ker \pi_n$ is locally finite (since an extension of two locally finite groups is still locally finite). To  see this,  let $g\in \operatorname{ker}{\pi}$ and let $|g|$ be its word metric with respect to the standard generating set. Since $b_i-b_{i-1}$ tends to infinity, the ball of radius $|g|$ around any vertex $v\in T$ sufficiently far from the root contains at most one branching point. Hence for every sufficiently far vertex $v\in T$, the ball of radius $|g|$ around $v$ coincides with a ball in some $\overline{T}_n$, for $n$ admissible. Since $g$ has trivial projection in $\overline{G}_n$ fore every admissible $n$, we deduce that $g$ fixes $v$, and thus it fixes all vertices sufficiently far from the root. It follows that $\operatorname{ker}(\pi)$ consists of permutations with finite support, hence it is locally finite. As noted, this shows that $\ker p$ is locally finite and concludes the proof of (i) under the assumption that $(d_i)$ is bounded.

\begin{figure}
  \begin{center}
    \vspace{-20pt}
    \includegraphics[width=0.8\textwidth]{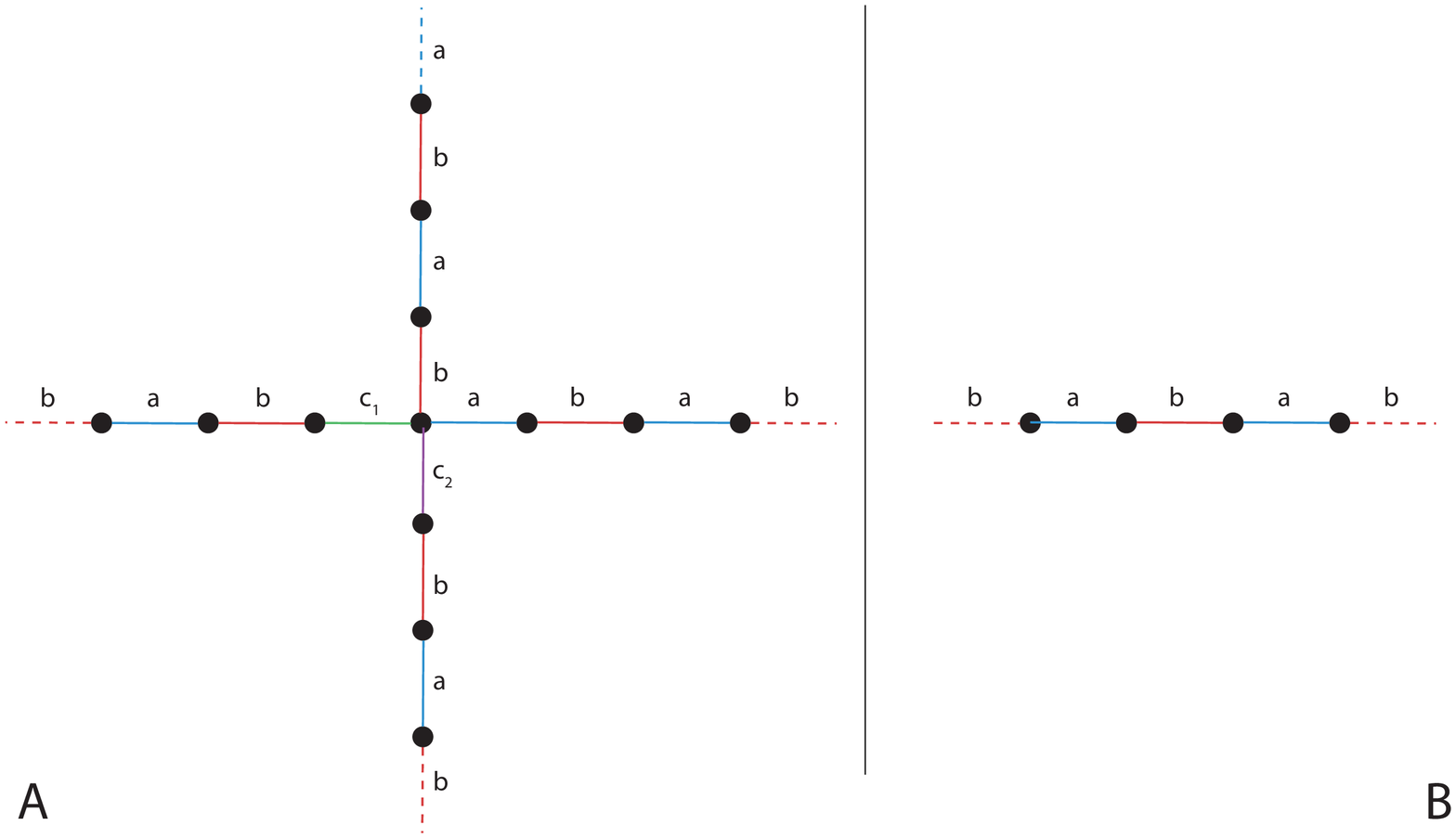}
  \end{center}
  \vspace{-10pt}
  \caption{A: the graph $\overline{T}_n$ for $n=4$. B: the graph $\hat{T}$\label{fig:not-bubble}}

  \begin{center}
    \includegraphics[width=0.8\textwidth]{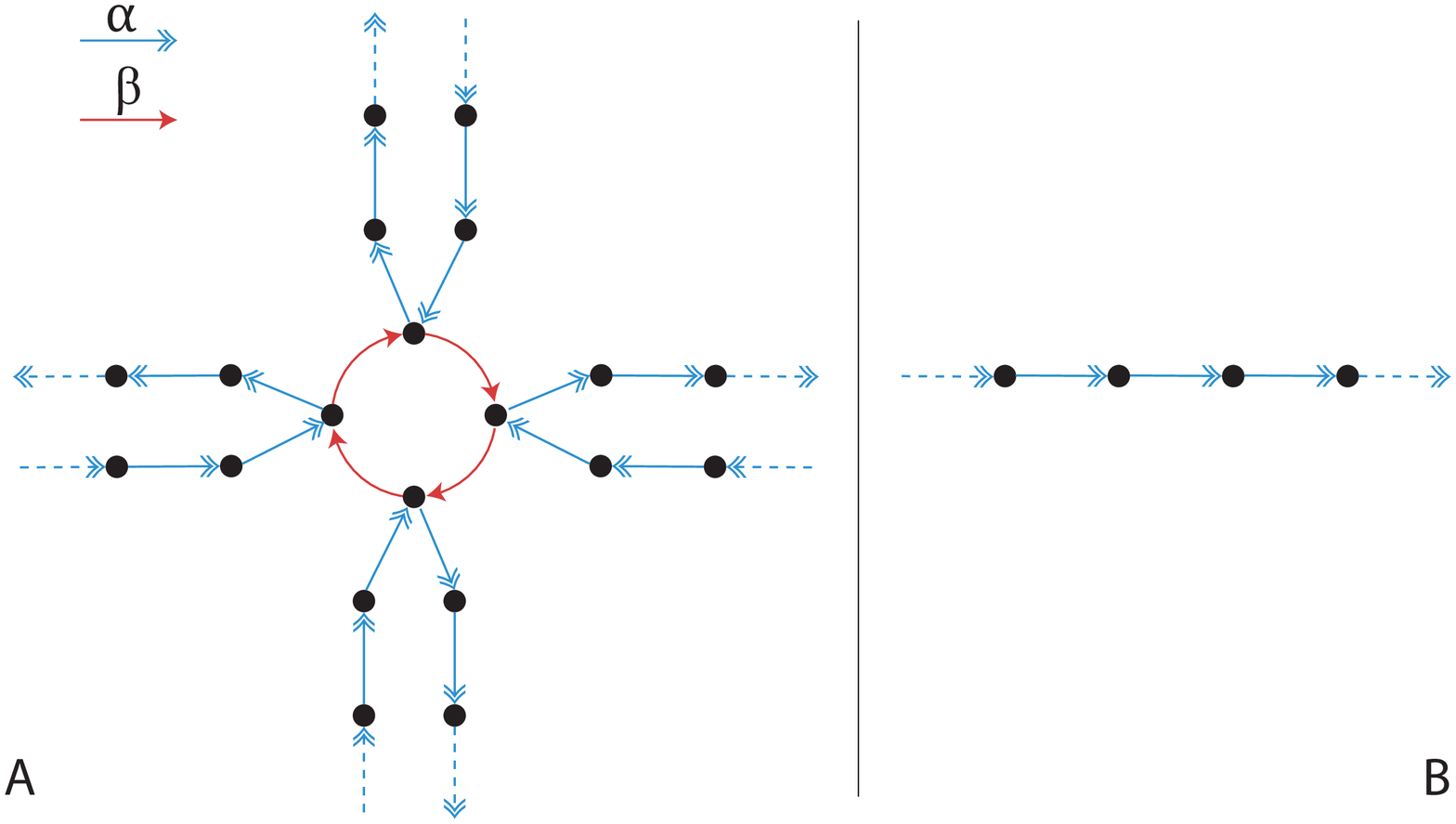}
  \end{center}
  \vspace{-10pt}
  \caption{A: the graph $\overline{\Theta}_n$ for $n=4$. B: the graph $\hat{\Theta}$\label{fig:bubble-son}}
\end{figure}

To remove this assumption, write $G$ as the ascending union of its finitely generated subgroups $G\vert_j$ generated by $S_j=\{a, b, c_1, \ldots c_j\}$, and let $T \vert_j$ be the graph obtained from $T$ by removing all edges labelled by $c_i, i> j$. Essentially the same argument as in the case of a bounded degree sequence shows that we have surjections $p|_j\colon G|_j\to D_\infty$ with locally finite kernel (a minor modification is needed since the graph $T\vert_j$ is no longer connected, hence strictly speaking it is not an element of $\sch(G\vert_j, S_j)$: consider instead the closure of the orbits of all its connected components and argue in a similar way). This shows that we have surjections $p|_j\colon G|_j\to D_\infty$ with locally finite kernel, mapping $a$ and $b$ to the standard generators of $D_\infty$ and each $c_i$ to 1. Thus they globally define a surjection $p\colon G\to D_\infty$ with locally finite kernel.

We now prove (ii). Assume first that the sequence $(d_i)$ is bounded. We consider the closure of $(\theta, o)$ in $\sch(\Gamma, S)$. Similarly to the previous case, the graphs in the closure have underlying unrooted graph of three types: the graph $\Theta$, the graphs of the form $\overline{\theta}_n$ for $n$ admissible (see Figure \ref{fig:bubble-son} A) and the graph $\hat{\Theta}$ (see Figure \ref{fig:bubble-son} B).

Denote by $\overline{\Gamma}_n$ the group defined by the graph $\overline{\theta}_n$, by $\overline{\pi}_n: \Gamma\to \hat{\Gamma}_n$ the corresponding projection, and by $\overline{\Gamma}$ the  image of $\Gamma$ into $\prod_n \overline{\Gamma}_n$ under the diagonal map $\pi=\prod_n\pi_n$. As in the case of $G$,  $\ker{\pi}$ is locally finite.  Closer inspection shows that the group $\overline{\Gamma}_n$ is isomorphic to $\Z/n\Z\wr \Z$, and the projection $\pi_n$ sends the generator $\alpha$ to a standard generator of $\Z$, and the generator $\beta$ to a standard generator of the cyclic lamp group over $0$. It follows that $\overline{\Gamma}$ is isomorphic to $\Z/\ell\Z\wr \Z$, where $\ell$ is the l.c.d. of all the admissible integers $n\in \N$. This concludes the proof for $\Gamma$ under the assumption that $(d_i)$ is bounded.

Let us assume now that $(d_i)$ is unbounded. In this case the argument for $\Gamma$ is  slightly different from the one used for $G$, since $\Gamma$ is  still finitely generated, and we may still work in the space $\sch(\Gamma, S)$. The closure of the orbit of $(\Theta, o)$ now contains the same graphs as in the bounded degree case, plus the  graph $\overline{\Theta}_\infty$ obtained by taking a limit of  graphs of the form $(\overline{\Theta}_n, x)$ when $n$ goes to $\infty$ and the root $x$ belongs to a branching cycle (in plain words,  $\overline{\Theta}_\infty$ consists of a bi-infinite oriented line labelled by $\beta$, to each vertex of which is glued a bi-infinite line labelled by $\alpha$). The group permutation group of $\overline{\Gamma}_\infty$ defined by its labelling is easily seen to be isomorphic to $\Z\wr \Z$. To conclude the proof, it is enough to show that the kernel of the natural surjection $\pi_\infty\colon \Gamma \to \overline{\Gamma}_\infty$ consists of finitely supported permutations of $\Theta$.  To this extent, let $g\in \operatorname{Ker}\pi_\infty$, and write $g=s_1\ldots s_n$ as a product of generators. Observe that $\overline{\Theta}_\infty$ covers all graphs of the form $\overline{\Theta}_n$, for all $n\in \N$ (the covering map is given by folding the line labelled by $\beta$ to an $n$-cycle, and identifing two $\alpha$-lines if the vertices of the $\beta$-line to which they are glued are identified).  If a vertex $v\in \Theta$ lies sufficiently far from the root, then the ball of a fixed radius $r$ around $v$ is isomorphic to a   ball of radius $r$ in a graph of the form $\overline{\Theta}_n$, for some $n\in \N$. Since $g\in \operatorname{Ker}\pi_\infty$, the path labelled by $s_1, \ldots s_n$ in this ball lifts to a closed path in $\overline{\Theta}_\infty$, hence it was already closed. This shows that $gv=v$. It follows that $g$ is a permutation with finite support, concluding the proof. \qedhere
\end{proof}





\end{document}